\documentclass[11pt]{article}
\usepackage{etoolbox}
\newbool{journal}
\setbool{journal}{false}

\usepackage{stmaryrd, extarrows}
\usepackage{amsthm}
\usepackage{bbm}
\usepackage{mathtools}
\usepackage{csquotes}
\usepackage{tikz}
\usetikzlibrary{positioning, shapes, calc}
\usepackage{subcaption}
\usepackage{comment}
\usepackage{booktabs}
\usepackage{orcidlink}

\usepackage{hyperref} % For hyperlinks in the PDF
\usepackage{breakurl}
\usepackage{cleveref} % cref and pref commands i need/make

\providecommand{\keywords}[1]{\textbf{\textit{Keywords:}} #1}

\newcommand{\braces}[1]{\lbrace #1 \rbrace}
\DeclareMathOperator{\Sol}{Sol}

\newcommand{\F}{$(F)$ }
\newcommand{\Ft}{$(F')$ }
\newcommand{\Fncspace}{\!\!}
\newcommand{\pref}[1]{(\ref{#1})}

\newtheorem{theorem}[]{Theorem}

\newtheorem{corollary}[]{Corollary}

\usepackage{authblk}

\begin{document}

\title{Fractional Chromatic Numbers \\from Exact Decision Diagrams}

                \author[1]{Timo Brand \orcidlink{0009-0004-3111-2045}}
                \affil[1]{\normalsize Technical University of Munich;
                School of Computation, Information and Technology;
                    Germany, \href{Timo.Brand@tum.de}{Timo.Brand@tum.de}}

                \author[2]{Stephan Held \orcidlink{0000-0003-2188-1559}}
                \affil[2]{\normalsize Research Institute for Discrete Mathematics and Hausdorff Center for Mathematics, University of Bonn,
                  Germany, \href{held@dm.uni-bonn.de}{held@dm.uni-bonn.de}} % ORCID:0000-0003-2188-1559 

                \maketitle

                %----------------------------------------------------------------------------------------
                %	Abstract
                %----------------------------------------------------------------------------------------
                \begin{abstract}
                  Recently, Van Hoeve proposed an algorithm for graph
                  coloring based on an integer flow formulation on decision diagrams for stable sets
                  \cite{van2021graph}.  We prove that
                  the solution to the linear flow relaxation on exact decision diagrams determines the fractional
                  chromatic number of a graph.
                  This settles the question whether the decision diagram formulation or the fractional chromatic number establishes
                  a stronger lower bound.
                  It also establishes that the integrality gap of the linear programming relaxation is $\mathcal{O}(\log n)$, where 
                  $n$ represents the number of vertices in the graph.
                  
                  We also conduct experiments using exact
                  decision diagrams and could
                  determine the chromatic number of \texttt{r1000.1c}
                  from the DIMACS benchmark set.
                  It was previously unknown and is one of the few
                  newly solved DIMACS instances in the last 10 years.
                \end{abstract}


        %----------------------------------------------------------------------------------------
        %	Introduction
        %----------------------------------------------------------------------------------------
        \section{Introduction}\label{sec:introduction}

        A (vertex) coloring of a graph $G=(V,E)$ assigns  a color to each vertex such that adjacent vertices have different colors.
        Thus, each set of vertices with the same color is a \textit{stable set} (also called \textit{independent set}) in $G$.
        The \textit{(vertex) coloring problem} is to compute a coloring with the minimum possible number of colors.
        This number is denoted by $\chi(G)$ and also  called the \textit{chromatic number} of $G$.

        Let $\mathcal{S}$ denote the set of all \textit{stable sets} in $G$.
        Then, solving the following   integer programming model yields the chromatic number \cite{mehrotra1996column}:
        \begin{equation}
          \label{VCIP}
          \begin{aligned}
            \chi(G) = \; &\text{min}  \sum_{S\in \mathcal{S}} z_S \\
            &\text{s.t.} \;\;  \sum_{S\in \mathcal{S}: j \in S} z_i \geq 1 \quad  \forall j \in V\\
            & z_S\in \lbrace 0, 1 \rbrace \quad \forall S \in \mathcal{S}\\
          \end{aligned}
          \tag{VCIP}
        \end{equation}

        It is a special case of the  \textit{set cover problem}, where the vertex set  $V$ has to be covered with a minimum number of stable sets.
        The linear relaxation  results in the \textit{fractional chromatic number} $\chi_f(G)$.

        \begin{equation}
          \label{VCLP}
          \begin{aligned}
            \chi_f(G) := \; &\text{min}  \sum_{S\in \mathcal{S}} z_I \\
            &\text{s.t.} \;\;  \sum_{S \in \mathcal{S}:j\in S}z_S  \geq 1 \quad  \forall j \in V\\
            & 0 \leq z_S \leq 1 \quad \forall S \in \mathcal{S}.\\
          \end{aligned}
          \tag{VCLP}
        \end{equation}

        Lov{\'a}sz  showed that $\chi(G)\le \mathcal{O}(\log n)\cdot\chi_f(G)$ \cite{lovasz1975ratio}.
        However for any  $\epsilon > 0$, approximating either $\chi_f(G)$ or
        $\chi(G)$ within a factor $n^{1-\epsilon}$ is NP-hard
        \cite{zuckerman2007linear}.
        The IP formulation (\ref{VCIP}) is the basis for several  branch-\&-price algorithms for vertex coloring
        \cite{mehrotra1996column,malaguti-monaci-toth,held2012maximum}.

        Recently, Van Hoeve \cite{van2021graph} proposed
        to tackle the graph coloring problem with decision diagrams.
        A (binary) decision diagram consists of an acyclic digraph with arc labels.
        It can represent the set of feasible solutions to an optimization problem $P$,
        e.g.\ the stable sets of a graph. They are called exact if they represent all solutions, e.g.\ all
        stable sets.
        Van Hoeve showed how to compute a graph coloring  using a constrained network  flow in the decision diagram.

        \subsection{Contributions}
        We show that the linear relaxation of the integral network flow in an exact decision diagram
        describes the fractional chromatic number, i.e.\ the linear relaxations arising
        from (\ref{VCLP}) and from exact decision diagrams lead to the same lower bound.
        Thus, fractional flows in relaxed decision diagrams provide fast lower bounds for
        the (fractional) chromatic number.

        Finally, we show that the exact decision diagrams are computationally
        efficient on dense instances.
        For the first time, we are able to  compute the chromatic number of \texttt{r1000.1c}
        \cite{johnson1996cliques}.
        To provide a provably correct solution, we employ exact arithmetic
        using SCIP-exact \cite{eifler2023computational}.
        In unsafe floating point arithmetic,
        we could also improve the best known lower bound of the instance \texttt{DJSC500.9}.

        The paper is organized as follows. In
        Section~\ref{sec:decision_diagrams}, we shortly describe
        decision diagrams for the stable set problem and how a graph
        coloring integer program based on such a decision diagram can
        be formulated (Section~\ref{sec:coloring-flow-ip}) as proposed by
        \cite{van2021graph}.  Then, in Section~\ref{sec:linear_relax}
        we prove that the solution to the linear relaxation of the
        integer program determines the fractional chromatic number.
        Section~\ref{sec:results_instance} contains experimental results with exact arithmetic
        on the instances of the DIMACS benchmark set, followed by Conclusions.

        \section{Decision Diagrams}
        \label{sec:decision_diagrams}
        Here, we recap how stable sets can be represented through decision diagrams as proposed in \cite{bergman2014optimization,van2021graph}.
        See also van Hoeve's recent tutorial on decision diagrams for optimization \cite{van2024introduction}.
        Consider the decision variables $X = \braces{x_1, x_2, \ldots, x_n}\in \{0,1\}^n$
        of an optimization problem $P$. 
        Decision diagrams are a possible way  to represent the solution space $\Sol(P)$.

        A \textit{decision diagram} consists of a  layered directed acyclic graph $D = (N, A)$.
        $D$ has $n+1$ layers $L_1, \ldots, L_{n+1}$, where $N = L_1\dot{\cup}\dots\dot{\cup}L_{n+1}$.
        Arcs  go only from nodes in one layer to nodes in the next layer.
        The first layer $L_1$ consists of a single \textit{root node} $r$. Likewise,
        the last layer $L_{n+1}$ consists only of a single \textit{terminal node} $t$.
        A layer $L_j$ ($1\le j\leq n$) is a collection of nodes of $D$. Layer $j$ is  associated
        with the decision variable $x_j\in X$.
        For a node $u\in L_j$ we denote its layer $j$ by $L(u)$.

        Furthermore, the decision diagram has arc labels $l:A \to \{0,1\}$.
        Arcs are called  $0$-arcs or  $1$-arcs depending on their labels.
        A label encodes if the decision variable of the head level is set to 0 or 1,
        as we will describe later.
        For each node, except the terminal node, we have exactly one outgoing $0$-arc
        and at most one outgoing $1$-arc.

        Each node and each arc must belong to a path from $r$ to $t$.
        Given the arcs $(a_1, a_2, \ldots, a_n)$ of an $r$-$t$ path,
        we can define a variable assignment of $X$ by setting $x_j = l(a_j)$ for $j=1, \ldots, n$.
        $\Sol(D)$ denotes the collection of variable assignments obtained by all $r$-$t$ paths.
        A decision diagram $D$ for problem $P$ is called \textit{exact}
        if $\Sol(P) = \Sol(D)$ and \textit{relaxed} if $\Sol(P) \subseteq \Sol(D)$.

        Systematic ways to compute such decision diagrams can be found in \cite{bergman2014optimization}.

        \subsection{Decision Diagrams for Stable Sets}
        Let $G = (V,E)$ be a graph with vertex set $V= \{v_1,\dots, v_n\}$.
        To encode the  stable set problem, $X$ contains a  decision variable for
        each vertex in $V$.
        Now, an (exact) decision diagram $D$ with labels $l$ for the stable set problem on  $G$
        consists  of $n+1$ layers.
        For each stable set $S$ in $G$, there is exactly one path $(a_1,\dots, a_n)$ in $D$
        with $l(a_i) = \mathbbm{1}_S(v_i)$ and vice versa, where $\mathbbm{1}_S$ is the incidence vector of $S$.

        \Cref{ddexample} shows a graph $G$ on 4 vertices and an
        exact decision diagram with 5 layers representing all stable sets.
        The top layer contains the root $r$ and the bottom layer the terminal $t$.
        1-arcs are drawn as solid lines and 0-arcs as dashed lines.

        Notice how the $1$-arcs on a path from the root to the terminal
        correspond to the vertices of a stable set in the graph. 
        Likewise for each stable set in $G$ we can find a corresponding path in $D$.
        Thus,  $\Sol(P) = \Sol(D)$.

        The node labels in the decision diagram show the set of
        \textit{eligible vertices} at each node $v\in N$, i.e.\ the vertices
        that can be still be added individually to the stable sets
        corresponding to the $r$-$v$ sub-paths.

        {
          \newcommand{\shiftvalue}{0.5cm}
          
          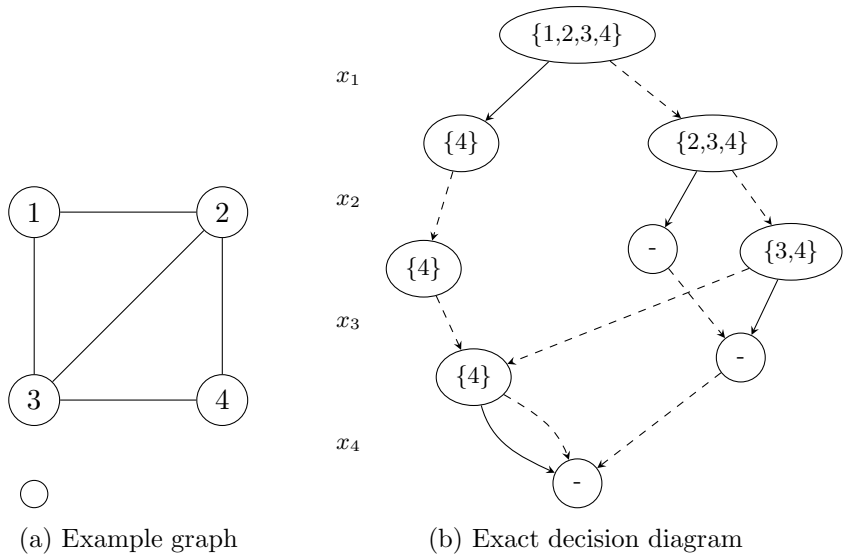
\begin{figure}[tb]
            
            \subcaptionbox{Example graph\label{ddexampleA}}[0.4\linewidth]{
              \vspace{0cm}
              \begin{tikzpicture}[every node/.style={circle, draw}, scale = 2.5]
                \node[opacity = 0] {};
                \begin{scope}[yshift = 0.5cm]
                  \node (1) at (0,1) {1};
                  \node (2) at (1,1) {2};
                  \node (3) at (0,0) {3};
                  \node (4) at (1,0) {4};
                \end{scope}

                \draw (1) -- (3)
                (3) -- (4)
                (4) -- (2)
                (2) -- (1)
                (3) -- (2);
              \end{tikzpicture}
            }
            \subcaptionbox{Exact decision diagram\label{ddexampleB}}[0.54\linewidth]{
              \begin{tikzpicture}[node distance=0.9cm and 0.45cm, scale = 1]\footnotesize
                \begin{scope}[local bounding box=graph]
                  \tikzset{%
                    zeroarrow/.style = {-stealth,dashed},
                    onearrow/.style = {-stealth,solid},
                    c/.style = {ellipse,draw,solid,minimum width=2em,
                      minimum height=2em},
                    r/.style = {rectangle,draw,solid,minimum width=2em,
                      minimum height=2em}
                  }

                  \node[c] (A1) {\{1,2,3,4\}};
                  \node[c] (B1) [below left=of A1] {\{4\}};
                  \node[c] (B2) [below right=of A1] {\{2,3,4\}};

                  \node[c] (C2) [below=of B1, xshift=-\shiftvalue] {\{4\}};
                  \node[c] (C3) [below left=of B2, xshift=\shiftvalue] {-};
                  \node[c] (C4) [below right=of B2, xshift=-\shiftvalue] {\{3,4\}};

                  \node[c] (D1) [below left=of C4, xshift=\shiftvalue] {-};
                  \node[c] (D2) [below right=of C2, xshift=-\shiftvalue] {\{4\}};

                  \node[c] (E1) [below=5.25cm of A1] {-};

                  \draw[onearrow] (A1) -- (B1);
                  \draw[zeroarrow] (A1) -- (B2);

                  %				\draw[onearrow] (B1) -- (C1);
                  \draw[zeroarrow] (B1) -- (C2);
                  \draw[onearrow] (B2) -- (C3);
                  \draw[zeroarrow] (B2) -- (C4);

                  \draw[zeroarrow] (C2) -- (D2);

                  \draw[zeroarrow] (C3) -- (D1);
                  \draw[onearrow] (C4) -- (D1);
                  \draw[zeroarrow] (C4) -- (D2);

                  \draw[zeroarrow] (D1) -- (E1);
                  \draw[onearrow] (D2) to [out=285, in=155] (E1);
                  \draw[zeroarrow] (D2) to [out=330, in=110] (E1);

                \end{scope}
                \foreach \X in {1,2,3, 4}
                         {\node at ([xshift=-0.5cm, yshift=4cm - 1.625*\X cm]graph.west) {$x_{\X}$}; }
              \end{tikzpicture}
            }
            \caption[An example graph and its corresponding exact decision diagram]{
              An example graph and its exact decision diagram.
            }\label{ddexample}
          \end{figure}
        }

        In the stable set case, nodes in a common layer with the same set of eligible vertices are
        called \textit{equivalent}. For general decision diagrams, two nodes $v,v'$ are equivalent if the two subgraphs induced by all $v$-$t$-paths and all $v'$-$t$-paths are isomorphic. For the stable set problem both notions are equivalent.
        Equivalent nodes can be contracted into a single node on
        that layer.
        This is used to reduce the size of a decision diagram for the stable set problem.
        Decision diagrams without equivalent nodes are called \textit{reduced decision diagrams}.

        Bergman et al. (\cite{bergman2014optimization}, Algorithm 1) proposed a top-down
        compilation technique to compute the exact reduced decision diagram $D$ for
        stable sets.

        Decision diagrams can have an exponential size.  They can be
        computed efficiently if the number of nodes per layer is small
        after contracting equivalent nodes, which is often the
        case for dense graphs.

       \subsection{Graph Coloring from Stable Set Decision Diagrams}
       \label{sec:coloring-flow-ip}
        For graph coloring, Van Hoeve \cite{van2021graph} proposed the  following  constrained integral network ﬂow problem \F on a stable set decision diagram $D = (N, A)$:
        \begin{align}
          (F) = \;\text{min } & \sum_{a \in \delta^{+}(r)} y_{a}          \label{flow:1}\\
          \quad \text{s.t. } & \sum_{a=(u, v) \mid L(u)=j, \ell(a)=1} y_{a} \geq 1
          & \forall j \in V               \label{flow:2}\\
          & \sum_{a \in \delta^{-}(u)} y_{a}-\sum_{a \in \delta^{+}(u)} y_{a}=0
          & \forall u \in N\setminus\{r,t\}               \label{flow:3}\\
          & y_{a} \in \lbrace 0,\ldots, n \rbrace & \forall a\in A        \label{flow:4}
        \end{align}
        The constraints \pref{flow:1}--\pref{flow:4} encode an integral
        $r$-$t$-flow that implicitly covers each original vertex
        through an activating arc \pref{flow:2}.
        The flow can be decomposed into paths, which correspond to stable sets.
        Minimizing the flow value corresponds to minimize the number of paths and, thus,
        to minimizing the number of stable sets in a stable set cover.

        Van Hoeve \cite[Theorem 2]{van2021graph} shows that \pref{flow:1}--\pref{flow:4} computes the chromatic number, if
        $D$ is an exact decision diagram.%
        \footnote{Formally, Van Hoeve proved this
        for the partitioning formulation of $(F)$, where \pref{flow:2}
        are equality constraints. Both formulations  are equivalent and he also uses
        the covering formulation in his implementation.}

        Relaxed decision diagrams might additionally contain paths  $(a_1,\dots, a_n)$ that do not represent stable sets.

        The main emphasis of Van Hoeve's work \cite{van2021graph} is  an iterative method
        to compute the chromatic number based on relaxed decision diagrams.
        While exact decision diagrams can have an exponential size,
        he begins with a relaxed decision diagram that approximates $\Sol(P)$.
        Solving the network flow problem \F on this relaxation yields a lower bound
        to the chromatic number and enough information to refine
        the relaxed decision diagram to become a better approximation.
        The refinement continues until an optimum coloring is found (or a time limit is reached).
        He also provides a class of instances where the refinement approach ends with a  
        polynomial size decision diagram and where an exact decision diagram has exponential size
        \cite[Theorem 7]{van2021graph}.
        Van Hoeve was able to report a new lower bound of 145 for instance \texttt{C2000.9}.

        %----------------------------------------------------------------------------------------
        %   Linear Relaxation
        %   ----------------------------------------------------------------------------------------
        \section{The Fractional Chromatic Number and Decision Diagrams}
        \label{sec:linear_relax}

        In this section we prove our main observation. For exact decision
        diagrams, the fractional chromatic number $\chi_f(G)$ is
        determined by the linear relaxation $(F')$ of $(F)$,  where the
        linear relaxation $(F')$ is defined as
        \begin{align}
          (F') = \;\text{min } & \sum_{a \in \delta^{+}(r)} y_{a}         \label{flowLP:1}\\
          \quad \text{s.t. } & \sum_{a=(u, v) \mid L(u)=j, \ell(a)=1} y_{a} \geq 1
          & \forall j \in V               \label{flowLP:2}\\
          & \sum_{a \in \delta^{-}(u)} y_{a}-\sum_{a \in \delta^{+}(u)} y_{a}=0
          & \forall u \in N\setminus\{r,t\}               \label{flowLP:3}\\
          & 0 \leq  y_{a} \leq n & \forall a\in A         \label{flowLP:4}
        \end{align}
        We have the following result.
        \begin{theorem}\label{thm:relaxation_characterization}
          Let $G=(V,E)$ be a graph and $D=(N,A)$  an exact stable set decision diagram for $G$.
          Then the linear relaxation $(F')$ is equivalent to \pref{VCLP}.
        \end{theorem}%
        \begin{proof}
          We show how optimum solutions can be transformed between the two problem formulations.
          
          Let  $(z_S)_{S\in \mathcal{S}}$ be an optimum basic solution of \pref{VCLP}.
          We transform it into a solution $y_a, a \in A$ of \Ft with the same objective value..

          For each $S \in \mathcal{S}$ with  $z_S > 0$, choose a path  $(a_1,\dots, a_n)$ 
          in the decision diagram, where $l(a_j) = 1$ if $j \in S$ and $l(a_j) = 0$ if $(j\not \in S)$.
          Since $S$ is a stable set and $D$ is exact, such a path must exist.
          Increase the flow along this path by $z_i$.
          
          As $z$ is an optimum basic solution, it uses at most $n$ stable sets with positive value.
          Since $0\le z\le 1$,  constraints \pref{flowLP:4} are satisfied.
          As we always augment the flow along $r$-$t$-paths, the flow condition \pref{flowLP:3} is also fulfilled.
          For each vertex $j\in V$ we have $\sum_{S\in \mathcal{S}:j\in S} z_S  \geq 1$.
          This implies  that at least one unit of flow  is sent through $1$-arcs
          in layer $j$ of the decision diagram, thus the solution satisfies \pref{flowLP:2} as well
          and is a valid solution of \Ft\Fncspace.
          By construction the objective values of both solutions coincide.

          For the other direction, let  $(y_{a})_{a\in A}$ be an optimum fractional solution of \Ft.
          It is an $r$-$t$ network flow that can be decomposed into flows along  $r$-$t$-paths,
          where no path is repeated.
          Let $P_1, \ldots, P_k$ be such a decomposition.
          Each path $P_i$ $(i\in[k])$ corresponds to a unique  stable set $S_i\in \mathcal{S}$ since $D$ is an exact decision diagram.
          We construct a solution $z$ of \pref{VCLP} by setting $z_{S_i}$ to the of flow sent along  path $P_i$ for $i \in [k]$ and $z_S=0$ for all stable sets that are not represented in the decomposition.
          By optimality,  $z \le 1$ and the variable bounds are satisfied.
          By \pref{flowLP:2}, the amount of flow on $1$-arcs in layer $j$ is at least $1$.
          This implies that for each vertex $j \in  V$, $\sum_{S\in \mathcal{S}:j\in S} z_S  \geq 1$, and $z$ is a feasible solution
            to \pref{VCLP} with the same objective value as $y$.

            We conclude that for an exact decision diagram of a graph $G$, \Ft
            computes the fractional chromatic number $\chi_f(G)$.
        \end{proof}

        From Theorem~\ref{thm:relaxation_characterization}, we can also conclude that the linear program \Ft
        always has an optimum solution with a polynomial-size support, given an exact decision diagram.

        \begin{corollary}
          Let $G=(V,E)$ be a graph and $D=(N,A)$  an exact stable set decision diagram for $G$,
          there is an optimum solution $y$ to \Ft with $|\{a \in A, y_a > 0\}| \le n^2$, where $n=|V|$.
        \end{corollary}
        \begin{proof}
          The transformation of a basic optimum solution of \pref{VCLP}
          results in an optimum solution $y$ that consists of at most $n$
          paths, where each path has at most $n$ edges.
        \end{proof}

        As the underlying decision diagram and, thus, the  support of a basic optimum solution to \Ft can have exponential size, the transformation of a solution $y$ for
        \Ft might  yield an optimum solution $z$ for \pref{VCLP} that is not a basic solution.

        A nice property of \Ft is that the formulation provides a lower bound
        on $\chi_f(G)$ for each relaxed decision diagram. Good bounds might be
        easier to compute in practice than for \pref{VCLP}, where
        the branch-\&-price algorithm can take many iterations with
        maximum-weight stable set problems in the pricing problem.

        Our result also shows that the (fractional) solution to \Ft of a decision diagram cannot
        improve upon the fractional chromatic number.
        
        Lov{\'a}sz  showed that the ratio between
        $\chi(G)$ and $\chi_f(G)$ is  $\mathcal{O}(\log n)$ \cite{lovasz1975ratio}.
        With this, we obtain the following corollary:
        \begin{corollary}\label{corollary}
          Given an exact decision diagram,
          the (integrality) gap between \F and \Ft is $\mathcal{O}(\log  n)$.
        \end{corollary}

        %----------------------------------------------------------------------------------------
        %   Results on instance r1000.1c
        %----------------------------------------------------------------------------------------
        \section{Computational Results with Exact Arithmetic}
        \label{sec:results_instance}

        We re-implemented and essentially verified the experimental
        results of Van Hoeve~\cite{van2021graph} using relaxed decision diagrams.
        Van Hoeve also showed for which instances
        the chromatic number can be computed efficiently within 1 hour
        of running time solving the ILP $(F)$ for the exact decision diagram.
        We repeated this experiment with our implementation,
        available at \url{https://github.com/trewes/ddcolors}.
        Corresponding scripts, logfiles and instructions for the experiments
        are archived at \cite{FK2/ZE9C3L_2024}.

        We ran our experiments on an AMD EPYC 7742 processor.
        The size  of the exact decision diagrams was limited to two  million nodes.
        To obtain numerically safe results, we used the exact ILP solver SCIP-exact
        for solving  $(F)$ and writing certificates~\cite{eifler2023computational}.
        The running time of SCIP-exact was limited to one hour.

        The results of these experiments are reported in \Cref{tab:results}
        for those DIMACS instances where the exact decision diagram was built with at most two million nodes.
        The columns under ``EDD'' show the ``size'' of the  exact decision diagrams and the running ``time'' to compute the decision diagram as well as computing an upper bound with the DSATUR heuristic \cite{brelaz:79}. 
        The column ``SCIP'' shows  the running time of SCIP-exact, where ``-'' stands for a timeout. The columns ``lb'' and ``ub'' show the
        lower and upper bounds reported by SCIP or DSATUR. They are in bold face if they reflect the chromatic number.
        The last row in the right column shows the total number of DIMACS instances, the number of instances for which the exact decision diagram
        could be computed, and the number of instances that could be solved using SCIP-exact.

        Observing good performance on \texttt{r1000.1c} and \texttt{DSJC500.9}, we further investigated these instances.
        We were able to compute the chromatic number of the DIMACS benchmark instance \texttt{r1000.1c}.
        Its decision diagram has 1,228,118 nodes, which is just above the limit of 1 million chosen by Van Hoeve.
        It takes SCIP-exact 3 hours and 11 minutes to determine $\chi(\text{\texttt{r1000.1c}}) = 98$.

        SCIP-exact was unable to improve the lower bound of \texttt{DSJC500.9},
        but CPLEX 12.6 reports a lower bound of $123.0013$ after running for 31 minutes using 4 threads,
        thus indicating $\chi(\text{\texttt{DSJC500.9}}) \geq 124 = \lceil 123.0013 \rceil$.
        Since CPLEX is not guaranteed to be numerically exact,
        we let CPLEX continue until a lower bound of $123.2121$ was found to increase the confidence in the result.

        \section{Conclusions}

        We showed that the fractional flow formulation for graph coloring
        applied to exact decision diagrams, introduced by Van Hoeve~\cite{van2021graph},
        determines the fractional chromatic number of a graph.
        The fractional lower bounds from relaxed decision diagrams presented in
        \cite{van2021graph} are thus lower bounds for the fractional chromatic number.
        It is an interesting alternative to using the set cover
        formulation, as it can be faster to compute on certain instances
        and solve instances that other approaches cannot.
        We used it to compute the chromatic number of \texttt{r1000.1c} for the first time
        and find an improved  (numerically unsafe) lower bound for \texttt{DSJC500.9}.

        %%Results table

        \begin{table}
            \caption{\footnotesize Performance of SCIP-exact on the instances where the exact decision diagram
              is built within the limit of two million nodes, using a one hour time limit for SCIP-exact. Running times are reported in seconds.}
            \label{tab:results}
            \setlength{\tabcolsep}{2.4pt}
            \scriptsize      
            \begin{tabular}[t]{@{\extracolsep{\fill}}lrrrcr}
            \toprule
            Instance & lb & ub & \multicolumn{2}{c}{~\qquad EDD} & SCIP \\
             &  &  & size & time &  time \\
            \midrule
            1-FullIns\_3 & \textbf{4} & \textbf{4} & 748 & 0.1 & 0.2 \\
            1-FullIns\_4 & 3 & 5 & 455,895 & 4.0 & - \\
            1-Insertions\_4 & 2 & 5 & 184,070 & 1.2 & - \\
            2-FullIns\_3 & \textbf{5} & \textbf{5} & 12,867 & 0.1 & 10.4 \\
            2-Insertions\_3 & 3 & 4 & 2,964 & 0.1 & - \\
            3-FullIns\_3 & 5 & 6 & 435,083 & 7.7 & - \\
            3-Insertions\_3 & 2.67 & 4 & 23,180 & 0.1 & - \\
            4-Insertions\_3 & 2 & 4 & 178,377 & 2.0 & - \\
            anna & 7 & 11 & 1,033,676 & 42.9 & - \\
            david & \textbf{11} & \textbf{11} & 37,030 & 0.1 & 76.8 \\
            DSJC125.5 & 15.73 & 21 & 668,423 & 3.4 & - \\
            DSJC125.9 & \textbf{44} & \textbf{44} & 9,869 & 0.1 & 2.6 \\
            DSJC250.9 & \textbf{72} & \textbf{72} & 80,682 & 0.1 & 176.5 \\
            DSJC500.9 & 122.48 & 161 & 794,089 & 0.9 & - \\
            DSJR500.1c & \textbf{85} & \textbf{85} & 145,777 & 0.1 & 11.7 \\
            fpsol2.i.1 & \textbf{65} & \textbf{65} & 8,296 & 0.1 & 0.9 \\
            fpsol2.i.2 & \textbf{30} & \textbf{30} & 10,168 & 0.1 & 1.3 \\
            fpsol2.i.3 & \textbf{30} & \textbf{30} & 10,258 & 0.1 & 1.3 \\
            huck & \textbf{11} & \textbf{11} & 1,078 & 0.1 & 0.3 \\
            inithx.i.1 & \textbf{54} & \textbf{54} & 15,805 & 0.5 & 1.8 \\
            inithx.i.2 & \textbf{31} & \textbf{31} & 24,589 & 0.3 & 4.6 \\
            inithx.i.3 & \textbf{31} & \textbf{31} & 24,551 & 0.3 & 4.3 \\
            jean & \textbf{10} & \textbf{10} & 5,252 & 0.1 & 1.3 \\
            miles1000 & \textbf{42} & \textbf{42} & 8,032 & 0.1 & 1.4 \\
            miles1500 & \textbf{73} & \textbf{73} & 4,008 & 0.1 & 0.4 \\
            miles250 & \textbf{8} & \textbf{8} & 2,813 & 0.1 & 0.9 \\
            miles500 & \textbf{20} & \textbf{20} & 15,273 & 0.1 & 8.8 \\
            \bottomrule
            \end{tabular}%
            \begin{tabular}[t]{lrrrcr@{}}
            \toprule
            Instance & lb & ub & \multicolumn{2}{c}{~\qquad EDD} & SCIP \\
             &  &  & size & time &  time\\
            \midrule
            miles750 & \textbf{31} & \textbf{31} & 13,154 & 0.1 & 3.4 \\
            mug88\_1 & 2 & 4 & 1,824,581 & 304.3 & - \\
            mulsol.i.1 & \textbf{49} & \textbf{49} & 2,488 & 0.1 & 0.3 \\
            mulsol.i.2 & \textbf{31} & \textbf{31} & 2,612 & 0.1 & 0.3 \\
            mulsol.i.3 & \textbf{31} & \textbf{31} & 2,622 & 0.1 & 0.3 \\
            mulsol.i.4 & \textbf{31} & \textbf{31} & 2,637 & 0.1 & 0.4 \\
            mulsol.i.5 & \textbf{31} & \textbf{31} & 2,650 & 0.1 & 0.3 \\
            myciel3 & \textbf{4} & \textbf{4} & 63 & 0.1 & 0.1 \\
            myciel4 & \textbf{5} & \textbf{5} & 460 & 0.1 & 1.1 \\
            myciel5 & 4.54 & 6 & 6,017 & 0.1 & - \\
            myciel6 & 3.91 & 7 & 263,509 & 1.1 & - \\
            queen5\_5 & \textbf{5} & \textbf{5} & 561 & 0.1 & 0.1 \\
            queen6\_6 & \textbf{7} & \textbf{7} & 2,687 & 0.1 & 0.8 \\
            queen7\_7 & \textbf{7} & \textbf{7} & 13,839 & 0.1 & 6.2 \\
            queen8\_12 & 4 & 13 & 1,710,874 & 47.3 & - \\
            queen8\_8 & 8.44 & 10 & 81,575 & 0.2 & - \\
            queen9\_9 & 5 & 12 & 486,777 & 4.5 & - \\
            r1000.1c & 95.83 & 100 & 1,228,118 & 1.4 & - \\
            r125.1 & \textbf{5} & \textbf{5} & 921 & 0.1 & 0.2 \\
            r125.1c & \textbf{46} & \textbf{46} & 4,008 & 0.1 & 0.3 \\
            r125.5 & \textbf{36} & \textbf{36} & 23,243 & 0.1 & 12.3 \\
            r250.1c & \textbf{64} & \textbf{64} & 20,323 & 0.1 & 1.5 \\
            r250.5 & 65 & 66 & 232,727 & 0.5 & - \\
            zeroin.i.1 & \textbf{49} & \textbf{49} & 2,770 & 0.1 & 0.3 \\
            zeroin.i.2 & \textbf{30} & \textbf{30} & 3,471 & 0.1 & 0.5 \\
            zeroin.i.3 & \textbf{30} & \textbf{30} & 3,458 & 0.1 & 0.5 \\
            \bottomrule
            \multicolumn{3}{c}{Instances/EDDs/Solved} & \multicolumn{2}{c}{137/53/36}\\
            \end{tabular}
        \end{table}

 \bibliographystyle{abbrv}
 \bibliography{references.bib}
\end{document}